\documentclass[plain]{baustms}

\citesort

\theoremstyle{cupthm}
\numberwithin{equation}{section}
\newtheorem{theorem}{Theorem}[section]
\newtheorem{cor}[theorem]{Corollary}

\newtheorem{prop}[theorem]{Proposition}

\newtheorem{rem}{Remark}[section]
\begin{document}
\runningtitle{Coronae graphs and their $\alpha$-eigenvalues}
\title{Coronae graphs and their $\alpha$-eigenvalues}
\author[1]{Muhammad Ateeq Tahir}
\address[1]{School of Mathematical Sciences, MOE-LSC, SHL-MAC, Shanghai Jiao Tong University Shanghai 200240, P. R. China. \email{ateeqtahir@sjtu.edu.cn}.}
\author[2]{Xiao-Dong Zhang}
\address[2]{School of Mathematical Sciences, MOE-LSC, SHL-MAC, Shanghai Jiao Tong University Shanghai 200240, P. R. China. \email{xiaodong@sjtu.edu.cn}.}

\authorheadline{Muhammad Ateeq Tahir and Xiao Dong Zhang}


\support{This work is supported by  the National Natural Science Foundation of China (No.11531001), the Montenegrin-Chinese Science and Technology Cooperation Project (No.3-12) and the Joint NSFC-ISF Research Program (jointly funded by the National Natural Science Foundation of China and the Israel Science Foundation (No. 11561141001).}

\begin{abstract}
Let $G_1$ and $G_2$ be two simple connected graphs.	 The invariant \textit{coronal} of graph is used in order to determine the $\alpha$-eigenvalues of four different types of graph equations that are $G_1 \circ G_2, G_1\lozenge G_1$ and the other two`s are  $G_1 \odot G_2$ and $G_1 \circleddash G_2$ which are obtained using the $R$-graph of $G_1$. As an application we construct infinitely many pairs of non-isomorphic $\alpha$-Isospectral graph.
\end{abstract}


\maketitle

\section{Introduction}
%
%
%
%
%
In this paper we consider a simple, undirected and connected graphs. Let $G$ be a graph with vertex set $V(G)$ and edge set $E(G)$ with $n$ vertices and $m$ edges respectively. For any vertex $u\in V(G)$, $d_G(u)$ is the degree of that vertex. The adjacency matrix $A(G)$ of the graph $G$ is defined by its entries $a_{uv}=1$ when $u\sim v$ in $G$ and $0$ otherwise, and the degree matrix $D(G)$ of $G$ is the diagonal matrix whose entries are the degree of the vertices of $G$. The Laplacian  matrix $L(G)$ is defined as $L(G)=D(G-A(G))$. The signless Laplacian matrix $Q(G)$ is defined as $Q(G)=D(G)+A(G)$. Two non-isomorphic graphs are said to be adjacency (Laplacian, signless Laplacian) Isospectral if they have same adjacency (Laplacian, signless Laplacian) eigenvalue.\\
A lot of new graphs are obtained using new operations such as disjoint union, the Cartesian product, Kronecker product, the corona, the edge corona, the subdivision-vertex join, the subdivision-edge join, the vertex neighborhood corona and the subdivision-edge neighborhood corona. The adjacency spectra (Laplacian and signless Laplacian) of these operations have been calculated in  \cite{Brouwer,CDS,CRS}.  So many infinite families of pairs of adjacency, Laplacian and signless Laplacian Isospectral graphs are generated using graph operation, see for example \cite{WZ,RF,LZ,LZ1,MM,BPS}. Before going further let have some basic definitions which come from \cite{Fharray,YW}.  Let $G_1$ and $G_2$ be two graphs on disjoint sets of $n_1$ and $n_2$ vertices with $m_1$ and $m_2$ edges respectively. The \textit{corona $G_1\circ G_2$} of $G_1$ and $G_2$ is defined as the graph obtained by taking one copy of $G_1$ and $n_1$ copies of $G_2$, then joining the $i^{th}$ vertex of $G_1$ to every vertex in the $i^{th}$ copy of $G_2$, see Figure(\ref{figure1}). So in that way $G_1\circ G_2$ has $n_1(n_2+1)$ vertices with $m_1+n_1(m_2+n_2)$ edges. Let $G_1$ and $G_2$ be two graphs on disjoint sets of $n_1$ and $n_2$ vertices, $m_1$ and $m_2$ edges, respectively. The \textit{edge coronal} $G_1 \lozenge G_2$ of two graphs is defined by taking one copy of $G_1$ and $m_1$ copies of $G_2$ and then joining two end vertices of the $i^{th}$ edge of $G_1$ to the every vertex in the $i^{th}$ copy of $G_2$, see Figure(\ref{figure1}). So that the edge corona $G_1 \lozenge G_2$ of $G_1$ and $G_2$ has $n_1 + m_1n_2$ vertices and $m_1 + 2m_1n_2 + m_1m_2$ edges.
\begin{figure}[h!]
	
	\includegraphics[width=60mm,scale=0.6]{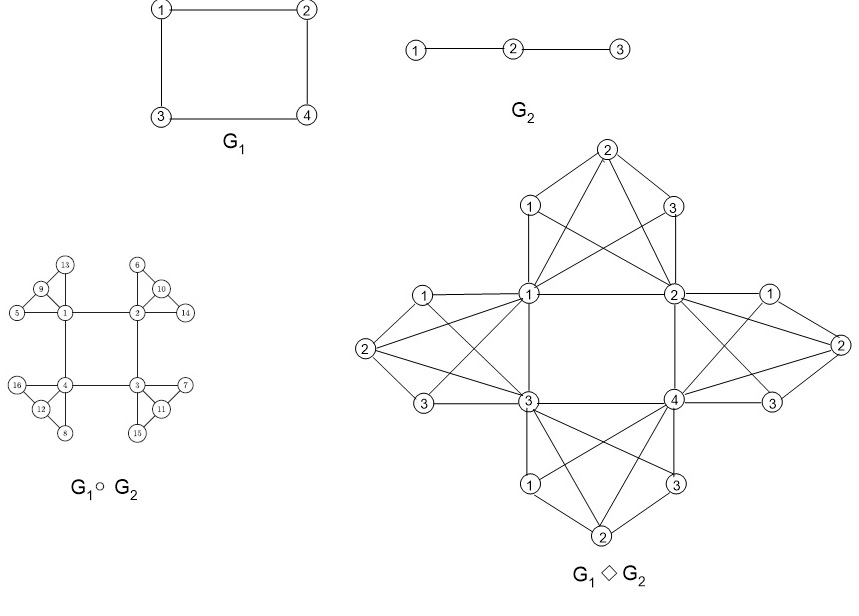}
	\centering
	\caption{ vertex corona and edge corona of $G_1=C_4$ and $G_2=P_3$.}
	\label{figure1}
\end{figure}\\
Similar to above we considered an another family of graph defined in \cite{LZhou}. For a graph $G$, Let $R(G)$ be the graph obtained from the graph by adding a vertex $v_e$ and joining $v_e$ to the end vertices of $e$ for each $e\in E(G)$. The graph $R(G)$ appeared in \cite{CDS} and named as the $R$-graph of $G$. In other words $R(G)$ is just the edge corona of $G$ and a singleton graph. Let $I(G)$ be the set of newly added vertices, i.e. $I(G)=V(R(G))\setminus V(G)$. Let $G_1$ and $G_2$ be two vertex disjoint graphs with $n_1$, $n_2$ vertices and $m_1$ and $m_2$ edges respectively. The $R$-\textit{vertexcorona} of $G_1$ and $G_2$, denoted by $G_1\odot G_2$, is the graph obtained from vertex disjoint $R(G_1)$ and $n_1$ copies of $G_2$ by joining the $i^{th}$ vertex of $V(G_1)$ to every vertex in the $i^{th}$ copy of $G_2$, see Figure(\ref{figure2}). So we have $n_1(1+n_2)+m_1$ vertices in $G_1\odot G_2$ with $3m_1+n_1(m_2+n_2)$ edges. The $R$-\textit{edgecorona} of $G_1$ and $G_2$, denoted by  $G_1\circleddash G_2$ is the graph obtained from vertex disjoint $R(G_1)$ and $|I(G_1)|$ copies of $G_2$ by joining the $i^{th}$ vertex of $I(G_1)$ to every vertex in the $i^{th}$ copy of $G_2$, see Figure(\ref{figure2}). Thus $G_1\circleddash G_2$ has $n_1+m_1(1+n_2)$ vertices with $m_1(3+m_2+n_2)$ edges. \begin{figure}[h!]
	\includegraphics[width=70mm,scale=0.7]{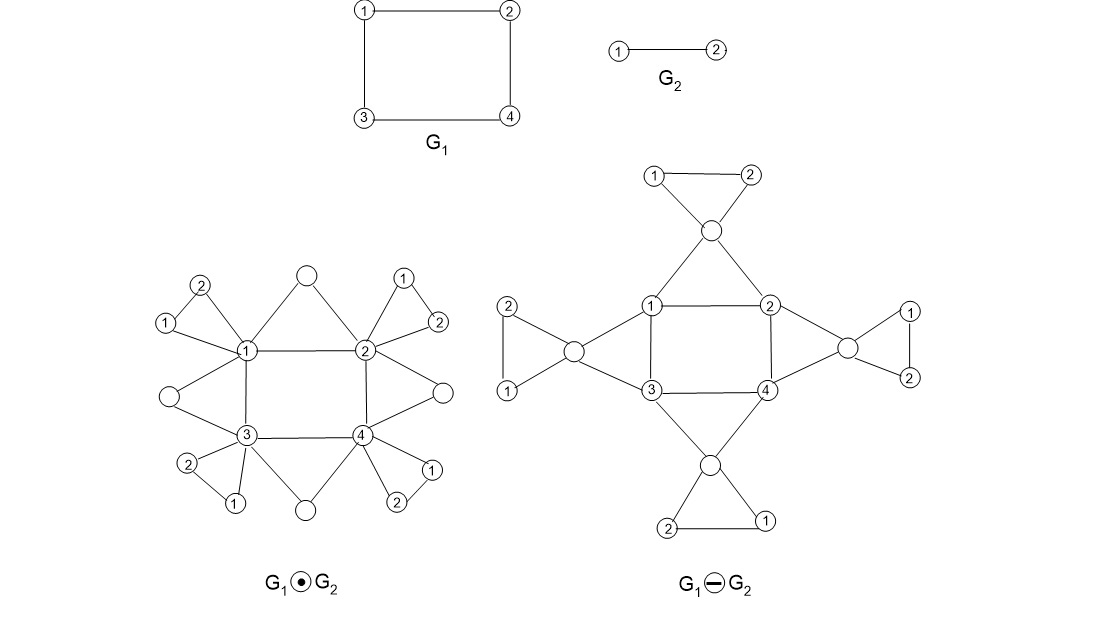}
	\centering
	\caption{ $R$-\textit{vertexcorona} and  $R$-\textit{edgecorona} of $G_1=C_4$ and $G_2=K_2$. }
	\label{figure2}
\end{figure}\\
Motivated by researcher, we discuss the spectrum of above mention graphs using a new family of matrices introduced in \cite{Vildo} which is in fact the convex combination of $A(G)$ and $D(G)$ defined by
\begin{equation}\label{mainequation}
A_\alpha(G)=\alpha D(G)+(1-\alpha)A(G), 0\leq \alpha \leq 1.
\end{equation}
So for $\alpha=0$ and $\alpha=\frac{1}{2}$ the equation(\ref{mainequation}) is similar to adjacency and signless Laplacian matrices of the graph. Let $M$ be an $n\times n$ matrix then $f_M(\theta)=(\theta I-M)$ is the characteristic polynomial  of $M$, where $\sigma(M)$ will be considered as set of eigenvalues of matrix $M$. For any graph $G$, $f_{A_\alpha}(\nu)$ is the $A_\alpha$ characteristic polynomial of graph $G$. Denote $\nu_1\geq ...\geq \nu_n$ be the $\alpha$-eigenvalues of the graph.  The rest of the paper is organized as fallows. In Section $2$, we provided some Preliminaries. In Section $3$, we computed the $A_\alpha$ characteristic polynomial of  $G_1\circ G_2$ when $G_1$ is any graph and $G_2$ is a regular(or complete bipartite) graph. In Section $4$, we computed the $A_\alpha$ characteristic polynomials of $ G_1\lozenge G_2$ for a regular graph $G_1$ and any arbitrary graph $G_2$. In Section $5$, and Section $6$ we computed the $A_\alpha$ characteristic polynomials of  of $G_1 \odot G_2$ and $G_1\circleddash G_2$. In last we discussed, as an application of these results and mentioned a pairs of non isomorphic $\alpha$-Isospectral graphs through which we can generate infinitely many family of $\alpha$-Isospectral graphs.
\section{Preliminaries}
The symbol $\mathbf{0_n}, \mathbf{j_n}$ (resp,$\mathbf{0_{mn}}$ and $\mathbf{J_{mn}}$) will stand for the column vector (resp, $m \times n$ matrices) consisting entirely of $0'$s and $1'$s. Also some standards notations $P_n,C_n$ and $K_{a,b}$ for the path, cycle on $n$ vertices  and complete bipartite graph on $n=a+b$ vertices. Let $G$ be a graph of order $n$ and $M$ be a graph of a matrix of $G$ the characteristic matrix $\theta I-M$ has determinant $det(\theta I-M)=f_M(\theta)\neq0$, so is invertible. The $M-$coronal of an $n\times n$ square matrix $M$ denoted by $\Gamma_M(\theta)$, defined in \cite{SG}, to be the sum of the entries of the matrix $(xI-M)^{-1}$, that is
\begin{equation} \label{coronal}
\Gamma_M(\theta)=\mathbf{j^T_n}(\theta I-M)^{-1} \mathbf{j_n}.
\end{equation}
It is known in \cite{SG} that if each row sum of $M$ equals to a constant $t$, then equation. (\ref{coronal}) has the form
\begin{equation}\label{coronal2}
\Gamma_M(\theta)=\frac{n}{\theta-t}.
\end{equation}
The vertex-egde incidence matrix $B(G)$ of a graph is the $(0,1)$-matrix with rows and columns indexed by the vertices and edges of $G$, respectively, such that $ue$-entry is $1$ if and only if vertex $u$ in incident with edge $e$. It is a known fact that $BB^T=Q(G)$ also their is an alternative way of writing equation(\ref{mainequation}) as
\begin{equation}\label{mainequation2}
Q(G)=\frac{1}{\beta}(A_\alpha(G)-(2\alpha-1)D(G)),
\end{equation} where $\beta=1-\alpha$ and $0\leq\alpha<1$. So we can write equation (\ref{mainequation2}) as $BB^T=\beta^{-1}(A_\alpha(G)-(2\alpha-1)D(G))$. We also need to recall some elementary results from Linear Algebra on the multiplication of Kronecker products and determinants of block matrices.
\begin{itemize}
	\item[i.] If $A$ is an $m \times n$ matrix and $B$ is a $p \times q$ matrix, then the Kronecker product $A \otimes B$ is the $mp \times nq$ block matrix.
	\item[ii.]  For any scalar $c$ we have $cA \otimes B= A\otimes cB =c(A \otimes B)$.
	\item[iii.] $(A \otimes B)^{-1}=A^{-1}\otimes B^{-1}$ and $det(A\otimes B)=det(A)^n det(B)^m$, if $A$ and $B$ are $n \times n$ and $m \times m$ matrices respectively.
	\item[iv.] If $A$ and $C$ are matrices of the same size, $B$ and $D$ are matrices of the same size, then $(A\otimes B)(C \otimes D)=(AC)\otimes (BD)$.
	\item[v.] If $D$ is invertible then $$det\begin{bmatrix}
	A & B\\
	C & D
	\end{bmatrix}= det(D)det(A-BD^{-1}C).$$
\end{itemize}
\section{$\alpha$-eigenvalue of $G_1 \circ G_2$}
Let $V(G_1)=\{u_1,...,u_{n_1}\}$ and $V(G_2)=\{v_1,...,v_{n_2}\}$ be two graphs on disjoint sets  of $n_1$ and $n_2$ vertices with $m_1$ and $m_2$ edges respectively. Then $G=G_1\circ G_2$ has a partition $V(G)=V(G_1)\cup V(G_2)$. Obviously the degrees of the vertices in $G=G_1\circ G_2$ are:
\begin{align*}
d_{G_1\circ G_2}(u_i) &=d_{G_1}(u_i)+n_2 for i=1,...,n_1,\\
d_{G_1 \circ G_2}(v_j) &= d_{G_2}(v_j)+1 for  j=1,...,n_2.\\
\end{align*}
Then $A_\alpha(G)$ can be written as.
$$A_\alpha(G)=\begin{bmatrix}
A_\alpha(G_1)+\alpha n_2I_{n_1} & \mathbf{j_{n_2}^T}\otimes \beta I_{n_1}  \\
\mathbf{j_{n_2}}\otimes \beta I_{n_1}& (A_\alpha(G_2)+\alpha I_{n_2}) \otimes I_{n_1}
\end{bmatrix},$$
where $\beta=1-\alpha$.
\begin{prop}\label{propcoronal1}
	Let $G_1$ and $G_2$ be two graphs on $n_1$ and $n_2$ vertices, respectively. Also let $\Gamma_{A_\alpha(G_2)}(\nu)$ be the $A_\alpha(G_2)$-coronal of $G_2$. Then the characteristic polynomial of $G= G_1 \circ G_2$ is $$ f_{A_\alpha(G)}(\nu)=( f_{A_\alpha(G_2)}(\nu-\alpha))^{n_1} \cdot f_{A_\alpha(G_1)}(\nu-\alpha n_2- \beta^2 \Gamma_{A_\alpha(G_2)}(\nu-\alpha) ),$$ In particularly, the $\alpha$-eigenvalue of $G$ is completely determined by the characteristics polynomial $f_{A_\alpha(G_1)}$ and$f_{A_\alpha(G_2)}$, and the $A_{\alpha(G_2)}$ of $G_2$.
\end{prop}
\begin{proof}
	The characteristic polynomial of $A_\alpha(G)$ can be calculated as fallows,
	$$f_{A_\alpha(G)}(\nu)=det(\nu I_{n_1(n_2+1)}-A_\alpha(G)).$$ Which implies
	\begin{align*}
	& =det\begin{bmatrix}
	\nu I_{n_1}- (A_\alpha(G_1)+\alpha n_2 I_{n_1}) &- \mathbf{j_{n_2}^T} \otimes \beta I_{n_1} \\
	- \mathbf{j_{n_2}} \otimes \beta I_{n_1}& \nu I_{n_1n_2} - (A_\alpha(G_2)+\alpha I_{n_2})\otimes I_{n_1}
	\end{bmatrix},\\
	& =det\begin{bmatrix}
	(\nu-\alpha n_2)  I_{n_1}- A_\alpha(G_1) &- \mathbf{j_{n_2}^T} \otimes \beta I_{n_1} \\
	- \mathbf{j_{n_2}} \otimes \beta I_{n_1}& ((\nu-\alpha)I_{n_2} - A_\alpha(G_2))\otimes I_{n_1}
	\end{bmatrix},\\
	& =det((\nu-\alpha)I_{n_2} - A_\alpha(G_2))\cdot det(B).
	\end{align*}	
	Where
	$$B=(\nu-\alpha n_2)  I_{n_1}- A_\alpha(G_1)-\beta^2 (\mathbf{j_{n_2}^T} \otimes I_{n_1})((\nu-\alpha)I_{n_2} - A_\alpha(G_2))^{-1}(\mathbf{j_{n_2}} \otimes I_{n_1}),$$ and
	$$det(B)=det\biggl((\nu-\alpha n_2)  I_{n_1}- A_\alpha(G_1)-\beta^2 (\mathbf{j_{n_2}^T} ((\nu-\alpha)I_{n_2} - A_\alpha(G_2))^{-1}\mathbf{j_{n_2}}) \otimes I_{n_1})\biggr).$$
	Using equation (\ref{coronal}) we get
	\begin{align*}
	det(B)&=det((\nu-\alpha n_2 -\beta^2 \Gamma_{A_\alpha(G_2)}(\nu-\alpha)) I_{n_1}- A_\alpha(G_1)),\\
	&=f_{A_\alpha(G_1)}(\nu -\alpha n_2 -\beta^2 \Gamma_{A_\alpha(G_2)}(\nu-\alpha)).
	\end{align*}	
	Hence the characteristic polynomial of $A_\alpha(G)$ is $$f_{A_\alpha(G)}(\nu)=( f_{A_\alpha(G_2)}(\nu-\alpha))^{n_1} \cdot f_{A_\alpha(G_1)}(\nu-\alpha n_2- \beta^2 \Gamma_{A_\alpha(G_2)}(\nu-\alpha) ),$$
	This completes the proof.
\end{proof}
Following Theorem is an easy consequence of Proposition(\ref{propcoronal1}) and equation(\ref{coronal2}) when $G_2$ is $k$-regular.
\begin{theorem}\label{theoremcoronal}
	Let $G_1$ be any graph on $n_1$ vertices and $G_2$ be a $k$-regular graph on $n_2$ vertices. Suppose $\sigma(G_1)=\{\nu_1, \nu_2,...,\nu_n\}$ and $\sigma(G_2)=\{k=\eta_1,\eta_2,...,\eta_n\}$ are the $\alpha$-eigenvalues of $G_1$ and $G_2$. Then the $\alpha$-eigenvalue of $G=G_1\circ G_2$ is given by:
	\begin{itemize}
		\item[i.] Two multiplicity- one eigenvalues $$\frac{\nu_i+k+\alpha(n_2+1)\pm\sqrt{\alpha^2 n_2^2+(k+\alpha-\nu_i)^2+2n_2(2+\alpha(\nu_i+\alpha-k-4))}}{2},$$ for each eigenvalue $\nu_i (i=1,...,n)$ of $A_\alpha(G_1)$.
		\item[ii.] The $\alpha$-eigenvalue $\eta_j+1$ with multiplicity $n_1$ for every non maximal $\alpha$-eigenvalue $\eta_j$ $(j=1,...,n_2-1)$.
	\end{itemize}
\end{theorem}
\begin{proof}
	Since $G_2$ is $k$-regular and so each row sum is $k$ using equation (\ref{coronal2}) that is, $$\Gamma_{A_\alpha(G_2)}(\eta-\alpha)=\frac{n_2}{(\eta-\alpha)-k},$$ $\eta= k+\alpha$ is the only pole in above equation. Which is equivalent to the maximal $\alpha$-eigenvalue $\eta-\alpha=k$ of $G_2$. Using Proposition(\ref{propcoronal1}), the $2n_1$  $\alpha$-eigenvalues are obtained by solving $$\eta-\alpha n_2-\beta^2\frac{n_2}{\eta-\alpha-k}=\nu_i.$$ This equation is quadratic in $\eta$, that is, $$\eta^2-(\alpha(n_2+1)+k+\nu_i)\eta+\alpha n_2(\alpha+k+\nu_i)-\beta^2 n_2+k\nu_i=0,$$ for each $\alpha$-eigenvalue $\nu_i$ of $A_\alpha(G_1)$. \\ Similarly the other $\eta_1(\eta_2-1)$ $\alpha$-eigenvalues are $\eta_j+1$ with multiplicity $n_1$ for every non maximum $\alpha$-eigenvalues $\eta_j$ , $j=1,...,n_2-1$.
\end{proof}
Now we will consider the $\alpha$-eigenvalue of $G=G_1\circ G_2$ when $G_1$ is any graph and $G_2=K_{a,b}$. If $a=b$, then $G_2$ is $a$-regular graph, which has been exhibited in Theorem(\ref{theoremcoronal}). It is always assumed that $a\neq b$ in the following Theorem.
\begin{theorem}\label{coronalofcomplete}
	Let $G_1$ be any graph on $n_1$ vertices and $G_2=K_{a,b}$ on $n_2=a+b$ vertices. Suppose that $\sigma(G_1)=\{\nu_1, \nu_2, ...,\nu_n\}$. Then the $\alpha$-eigenvalue of $G=G_1\circ G_2$ consists preciely of:
	\begin{itemize}
		\item[i.] The $\alpha$-eigenvalue $\alpha(a+1)$ with multiplicity $n_1(b-1)$.
		\item[ii.] The $\alpha$-eigenvalue $\alpha(b+1)$ with multiplicity $n_1(a-1)$.
		\item[iii.] For each $\alpha$-eigenvalue $\nu_i$ ($i=1,2,...,n_1$) of $A_\alpha(G)$, $\eta_{i1}+\alpha, \eta_{i2}+\alpha, \eta_{i2}+\alpha$ are three $\alpha$-eigenvalue of $G$ where $\eta_{i1}, \eta_{i2}, \eta_{i3}$ are the three roots of polynomial
		\begin{multline*}
		\eta^3+\eta ^2(-\nu _i-2 \alpha  n_2)+\eta (a (2 \alpha -1) b+n_2(\alpha (\nu _i+\alpha  n_2)-1))-2ab\\+\alpha  n_2 (a (1-2 \alpha ) b+n_2)+a (1-2 \alpha ) b \nu _i=0.
		\end{multline*}
	\end{itemize}
\end{theorem}
\begin{proof}
	We first compute the $A_\alpha(G_2)$-coronal of $G_2=K_{a,b}$, Let $$A_\alpha(G_2)=\begin{bmatrix}
	\alpha b I_a& \beta \mathbf{J}\\
	\beta \mathbf{J} & \alpha a I_b
	\end{bmatrix}.$$Consider the diagonal matrix $$X=\begin{bmatrix}
	(\eta -\alpha a+ \beta b) I_a&  \mathbf{0}\\
	\mathbf{0} & (\eta- \alpha b + \beta a )I_b
	\end{bmatrix}$$
	By a simple calculation we get $(\eta I-A_{\alpha}(G_2))X\mathbf{j_{n_2}}=\eta^2-\alpha(a+b)\eta+(2\alpha-1)ab\mathbf{j_{n_2}}$. Hence $$\Gamma_{A_{\alpha}(G_2)}(\eta)=\mathbf{j_{n_2}^T}(\eta I-A_{\alpha}(G_2))\mathbf{j_{n_2}}=\frac{ \eta(a+b)-\alpha(a^2+b^2)+2(1-\alpha)ab}{\eta^2-\alpha(a+b)\eta+(2\alpha-1)ab},$$ $$=\frac{ \eta n_2-\alpha n_2^2+2ab}{\eta^2-\alpha n_2 \eta+(2\alpha-1)ab}.$$ As the $\alpha$-eigenvalue of $K_{a,b}$ is $\sigma(K_{a,b})=\{\frac{\alpha n+\sqrt{\alpha^2n^2+4ab(1-2\alpha)}}{2},[\alpha a]^{b-1}$,\\$[\alpha b]^{a-1},\frac{\alpha n-\sqrt{\alpha^2n^2+4ab(1-2\alpha)}}{2}\}$. Where the two poles of $\Gamma_{A_{\alpha}(G_2)}(\eta-\alpha)$ are $\eta=\alpha + \frac{\alpha n+\sqrt{\alpha^2n^2+4ab(1-2\alpha)}}{2}$ and $\eta = \alpha+\frac{\alpha n-\sqrt{\alpha^2n^2+4ab(1-2\alpha)}}{2}$. Using Proposition(\ref{propcoronal1}), the $\alpha$-eigenvalue of $G$ is given as $\alpha (a+1), \alpha(b+1)$ with multiplicity $n_1(b-1)$ and $n_1(a-1)$ respectively. While the other $3n_1$ eigenvalues are obtained by solving: $$\eta - \alpha n_2 - \Gamma_{A_{\alpha}(G_2)}(\eta-\alpha)=\nu_i,$$ that is\\ $$\eta - \alpha n_2- \frac{(\eta-\alpha)n_2-\alpha n_2^2+ 2 ab}{(\eta-\alpha)^2-\alpha n_2(\eta-\alpha)+(2\alpha-1)ab}=\nu_i.$$ \\ This equation is of degree three in $\eta$, that is,
	\begin{multline*}
	2 a b (1 - 2 \alpha) + (a^2 + b^2) \alpha+ n_2 \alpha -n_2 \alpha^3 - n^2 \alpha^3 - a b n_2 \alpha (-1 +2 \alpha) + \eta^3 + \eta^2 (-2 \alpha \\ -2 n_2 \alpha- \nu_i) - \alpha^2 \nu_i + n \alpha^2 \nu_i +a b (-1 + 2 \alpha) \nu_i + \eta (-n_2 + \alpha^2 + 3 n_2 \alpha^2 + n^2 \alpha^2 \\+ a b (-1 + 2 \alpha) +2 \alpha \nu_i - n_2 \alpha \nu_i)=0.
	\end{multline*}
	Simplifying above will lead us to,
	\begin{multline*}
	a^2 \alpha + \eta^3 -a b (-1 + 2 \alpha) (2 +n_2 \alpha - \nu_i) + \eta^2 (-2 (1 +n) \alpha - \nu_i) + \alpha (b^2 + n - n \alpha^2 -n_2^2 \alpha^2 \\- \alpha \nu_i +n_2 \alpha \nu_i) + \eta (n^2 \alpha^2 + a b (-1 + 2 \alpha) + \alpha (\alpha + 2 \nu_i) + n_2 (-1 + 3 \alpha^2 - \alpha \nu_i))=0.
	\end{multline*}
	
\end{proof}
As an application of the Theorem(\ref{coronalofcomplete}), we constructed pairs of $\alpha$-eigenvalue of graphs. By Proposition (\ref{propcoronal1}), the $\alpha$-eigenvalues of $G_1 \circ G_2$ is completely determined by the characteristic polynomials $f_{A_\alpha(G_1)}$ and $f_{A_\alpha(G_2)}$ and the $A_\alpha(G_2)$-coronal of $G_2$. Which leads us to following result.
\begin{cor}\label{corr1}
	Let $G_1$ and $G_2$ are two non-isomorphic $\alpha$-Isospectral graph, and $H$ be any graph. Then,
	\begin{itemize}
		\item[i.] $G_1 \circ H$ and $G_2 \circ H$ are non-isomorphic $\alpha$-Isospectral graphs.
		\item[ii.] $H \circ G_1$ and $H  \circ G_2 $ are non-isomorphic $\alpha$-Isospectral graphs whenever $\Gamma_{A_{\alpha}(G_1)}=\Gamma_{A_{\alpha}(G_2)}$.
	\end{itemize}
\end{cor}
\section{$\alpha$-eigenvalue of $G_1  \lozenge G_2$}
Let the vertex set of a graph $G_1$ and $G_2$ be $V(G_1) = \{u_1, u_2,...,u_{n_1}\}$ and $V(G_2) = \{v_1,v_2,...,v_{n_2}\}$, respectively. Where the edge set of graph $G_1$ is $E(G_1) = \{e_1,e_2,...,e_{m_1}\}$. The vertex-edge
\textit{incidence} matrix $B_1(G) = (b_{ue} )$ is an $n_1 \times m_1$ matrix with entry $b_{ue_j} = 1$ if the vertex $u$ is incident the edge $e_j$ and $0$ otherwise. Let $G_1$ be $k_1$ regular graph and $G_2$ be any graph. Then the graph $G= G_1 \lozenge G_2$ has the partition $V(G)=V(G_1)\cup V(G_2)$. The degree of the vertices in $G_1 \lozenge G_2$ are:
\begin{align*}
d_{G_1\lozenge G_2}(u_i) &=d_{G_1}(u_i)+k_1n_2 for i=1,...,n_1,\\
d_{G_1 \lozenge G_2}(v_j) &= d_{G_2}(v_j)+2 for  j=1,...,n_2.\\
\end{align*}
Then using equation (\ref{mainequation}), the $A_\alpha$ matrix of $G=G_1 \lozenge G_2$ is written as: $$A_\alpha(G)= \begin{bmatrix}
A_\alpha(G_1) +\alpha k_1n_2 I_{n_1}& B_1 \otimes \beta \mathbf{j_{n_2}^T}\\
B_1^T \otimes \beta \mathbf{j_{n_2}}& I_{m_1}\otimes (2\alpha I_{n_2}+A_\alpha(G_2))
\end{bmatrix}$$ where $\beta=1-\alpha$.
\begin{rem}\label{remark4}
	Using fact $BB^T=Q(G)$ and equation (\ref{mainequation2}). If a graph $G$ is $k$ regular then $BB^T=\beta^{-1}(A_\alpha(G)-(2\alpha-1)kI)$.
\end{rem}
\begin{theorem}\label{edgecoronal}
	Let $G_1$ be an $k_1$ regular graph $n_1$ vertices, $m_1$ edges and $G_2$ be any graph with $n_2$ vertices, $m_1$ edges. Also let $\Gamma_{A_\alpha(G_2)}(\nu)$ be the $A_\alpha(G_2)$-coronal of $G_2$. if $\nu$ is not a pole of $\Gamma_{A_\alpha(G_2)}(\nu-2\alpha)$.
	\begin{multline*}
	f_{A_\alpha(G)}(\nu)=(f_{A_\alpha(G_2)}(\nu-2\alpha))^{m_1}\cdot f_{A_\alpha(G_1)}  \left ( \frac{\nu-\alpha k_1 n_2+\beta k_1(2\alpha-1)}{1+\beta \Gamma_{A_\alpha(G_2)}(\nu-2\alpha)} \right )\cdot\\\left [ 1+\beta\Gamma_{A_\alpha(G_2)}(\nu-2\alpha) \right ]^{n_1}.
	\end{multline*}
	
\end{theorem}
\begin{proof}The characteristic polynomial of $A_\alpha(G)$ can be calculated as fallows,
	$$f_{A_\alpha(G)}(\nu)=det(\nu I_{n_1+m_1n_2}-A_{\alpha}(G)).$$ So we have,
	\begin{align*}
	&= det \begin{bmatrix}
	\nu I_{n_1}-(A_\alpha(G_1) +\alpha k_1n_2 I_{n_1})& - B_1 \otimes \beta \mathbf{j_{n_2}^T}\\
	- B_1^T \otimes \beta \mathbf{j_{n_2}}& \nu I_{m_1n_2} - I_{m_1}\otimes (2\alpha I_{n_2}+A_\alpha(G_2))
	\end{bmatrix},\\
	&= det \begin{bmatrix}
	(\nu -\alpha k_1n_2)I_{n_1}-A_\alpha(G_1)& - B_1 \otimes \beta \mathbf{j_{n_2}^T}\\
	- B_1^T \otimes \beta \mathbf{j_{n_2}}& I_{m_1}\otimes ((\nu-2\alpha) I_{n_2}-A_\alpha(G_2))
	\end{bmatrix},\\
	&=det(I_{m_1}\otimes (\nu -2 \alpha)I_{n_2}-A_\alpha(G_2))\cdot det(S).
	\end{align*} where $S=((\nu-\alpha k_1 n_2)I_{n_1} -A_\alpha(G_1)-B_1B_1^T\otimes \beta^2(\mathbf{j^T_{n_2}}((\nu-2\alpha)I_{n_2}-A_\alpha(G_2))^{-1}\mathbf{j_{n_2}}))$.\\
	Using Remark(\ref{remark4}) and equation (\ref{coronal2}) we get, 	
	\begin{align*}
	S &=((\nu-\alpha k_1 n_2)I_{n_1} -A_\alpha(G_1)-(A_\alpha(G_1)-(2\alpha-1)I_{n_2}) \otimes \beta \Gamma_{A_{\alpha}(G_2)}(\nu -2\alpha)),\\
	&=((\nu-\alpha k_1 n_2+\beta(2\alpha-1) \Gamma_{A_{\alpha}(G_2)}(\nu -2\alpha))I_{n_1} -(1+ \beta \Gamma_{A_{\alpha}(G_2)}(\nu -2\alpha)A_\alpha(G_1)).\\
	det(S) &=det(((\nu-\alpha k_1 n_2+\beta(2\alpha-1) \Gamma_{A_{\alpha}(G_2)}(\nu -2\alpha))I_{n_1} -(1+ \beta \Gamma_{A_{\alpha}(G_2)}(\nu -2\alpha)A_\alpha(G_1))).\\
	&=f_{A_\alpha(G_1)}  \left ( \frac{\nu-\alpha k_1 n_2+\beta k_1(2\alpha-1)}{1+\beta \Gamma_{A_\alpha(G_2)}(\nu-2\alpha)} \right )\left [ 1+\beta\Gamma_{A_\alpha(G_2)}(\nu-2\alpha) \right ]^{n_1}.
	\end{align*}	
	Here completes the proof.
	
\end{proof}
\begin{theorem}
	Let $G_1$ be $k_1$-regular graph on $n_1$ vertices and $G_2$ be a $k_2$-regular graph on $n_2$ vertices. Suppose $\sigma(G_1)=\{k_1=\nu_1, \nu_2,...,\nu_n\}$ and $\sigma(G_2)=\{k_2=\eta_1,\eta_2,...,\eta_n\}$ are the $\alpha$-eigenvalues of $G_1$ and $G_2$ respectively. Then the $\alpha$-eigenvalue of $G=G_1\lozenge G_2$ is given by:
	\begin{itemize}
		\item[i.]  The $\alpha$-eigenvalue $\eta_j+2\alpha$ with multiplicity $m_1$ for every non maximal $\alpha$-eigenvalue $\eta_j$ ($j=1,...,m_1$) of $A_\alpha(G_2)$.
		\item[ii.] Two multiplicity one $\alpha$-eigenvalues by solving,\\ $\eta^2 + \eta (-2 \alpha -
		k_2 - k_1 (1 - 3 \alpha + 2 \alpha^2 + \alpha n_2) - \nu_i)+
		k_1 (2 \alpha + k_2) (1 - 3 \alpha +2 \alpha^2 + \alpha n_2)  + (2 \alpha + k_2 + (-1 + \alpha) n_2) \nu_i=0$, \\ for each $\alpha$-eigenvalue $\nu_i$, for ($i=1,...,n_1$), $A_\alpha(G_1)$, and
		\item[iii.] The $\alpha$-eigenvalue $k_2+2\alpha$ with multiplicity $m_1-n_1$(if possible).
	\end{itemize}
\end{theorem}
\begin{proof}
	As $G_2$ is $k_2$ regular, each row sum of $A_\alpha(G_2)$ is $k_2$, using equation (\ref{coronal2}) $$\Gamma_{A_\alpha(G_2)}(\eta-2\alpha)=\frac{n_2}{\eta-2\alpha-k_2}.$$ The only pole in above equation is $\eta=2\alpha+k_2$, which is equivalent to the maximal $\alpha$-eigenvalue $\eta-2\alpha=k_2$ of $G_2$. Now suppose that $\eta$ is not the pole in above equation. Then by Theorem(\ref{edgecoronal}), one has
	\begin{itemize}
		\item[i.] The $m_1(n_2-1)$ $\alpha$-eigenvalues are $\eta_j+2\alpha$ with multiplicity $m_1$ for every non maximal $\alpha$-eigenvalue $\eta_j$, (for $j=1,...,n_2-1$), of $A_\alpha(G_2)$ and
		\item[ii.] For each $\alpha$-eigenvalue $\nu_i$, for ($i=1,...,n_1$), of $A_\alpha(G_1)$. The $2n_1$ eigenvalues are obtained by solving, which is quadratic in $\eta$, $$\eta-\alpha k_1 n_2+\beta k_1(2\alpha-1)=\nu_i (1+\beta\frac{n_2}{\eta-2\alpha-k_2}),$$
		\begin{multline*}
		\eta^2 - 2 (1 - \alpha) \alpha (-1 + 2 \alpha) k_1 - (1 - \alpha) (-1 + 2 \alpha) k_1k_2 + 2 \alpha^2 k_1 n_2 + \alpha k_1 k_2 n_2\\ + \eta (-2 \alpha + (1 - \alpha) (-1 + 2 \alpha) k_1 - k_2 - \alpha k_1 n_2 - \nu_i + 2 \alpha \nu_i + 	k_2 \nu_i - (1 - \alpha) n_2\nu_i=0,
		\end{multline*}
		\begin{multline*}
		\eta^2 + k_1 (2 \alpha + k_2) (1 - 3 \alpha + 2 \alpha^2 + \alpha n_2) + \eta (-2 \alpha - k_2 - k_1 (1 - 3 \alpha + 2 \alpha^2\\ + \alpha n_2) -\nu_i) + (2 \alpha + k_2 + (-1 + \alpha) n_2) \nu_i=0.
		\end{multline*}
	\end{itemize}
	We get $m_1(n_2-1)+2n_1$, $\alpha$-eigenvalues of $G$. The other $n_1+m_1n_2-(m_1(n_2-1)+2n_1)=m_1-n_1$ $\alpha$-eigenvalues of $G$ must come from the only pole $\eta=k_2+2\alpha$ of  $\Gamma_{A_\alpha(G_2)}(\eta-2\alpha)$.
\end{proof}

Next we will give a little more description of $\alpha$-eigenvalue of $G_1\lozenge G_2$ when $G_1$ is $k_1$-regular and $G_2=K_{a,b}$.
\begin{theorem}\label{edgelabelcoronal2}
	Let $G_1$ be $k_1$-regular graph on $n_1$ vertices and $G_2=K_{a,b}$ on $n_2=a+b$ vertices and $m_2$ edges. Suppose $\sigma(G_1)=\{k_1=\nu_1, \nu_2,...,\nu_n\}$. Then the $\alpha$-eigenvalue of $G=G_1\lozenge G_2$ is given by:
	\begin{itemize}
		\item[i.] The $\alpha$-eigenvalue $\alpha(a+2)$ with multiplicity $m_1(b-1)$.
		\item[ii.] The $\alpha$-eigenvalue $\alpha(b+2)$ with multiplicity $m_1(a-1)$.
		\item[iii.] For each $\alpha$-eigenvalue $\nu_i$ ($i=1,2,...,n_1$) of $A_\alpha(G_1)$, $\eta_{i1}+2\alpha, \eta_{i2}+2\alpha$ and $\eta_{i3}+2\alpha$ are three $\alpha$-eigenvalues of $G$ where $\eta_{i1}, \eta_{i2}$ and $\eta_{i3}$ are three roots of cubic polynomial polynomial in $\eta$
		\begin{multline*}
		\eta^3+(((3-2 \alpha ) \alpha-\nu _i -1) k_1-\alpha  (k+1) n_2)\eta ^2+\eta (a (2 \alpha -1) b+n_2 ((2 \alpha -1) \nu _i \\ +(\alpha -1) (2 \alpha -1) \alpha  k_1+\alpha ^2 k n_2))-\nu _i (-(a b+(\alpha -1) \alpha  n_2^2))a (1-2 \alpha ) \alpha  b k n_2 \\-a (1-2 \alpha )^2 (\alpha -1) b k_1=0.
		\end{multline*}
		\item[iv.] The $\alpha$-eigenvalue $\{\frac{\alpha (n+4) \pm \sqrt{\alpha^2n^2+4ab(1-2\alpha)}}{2}\}$ with multiplicity $m_1-n_1$(if possible).
	\end{itemize}
	
\end{theorem}
\begin{proof} Using Theorem (\ref{coronalofcomplete})and Theorem (\ref{edgecoronal})) we can easily deduce $i, ii$ and $iv$.
	\begin{itemize}
		\item[$\bullet$] Solving $\eta-k_1(1-3\alpha+2\alpha^2+\alpha n_2)=\nu_i (1+\beta \Gamma_{A_{\alpha}(G_2)}(\eta-2\alpha))$,   we get a  little ugly equation
		\begin{multline*}
		\eta ^3+\eta ^2 (-4 \alpha -k_1 (2 \alpha ^2-3 \alpha +\alpha  n_2+1)-\nu +\alpha  n_2) +\eta(2 a (2 \alpha -1) b+4 \alpha  (\alpha +\nu )\\-\alpha  k_1 (n_2-4) (2 \alpha ^2-3 \alpha +\alpha  n_2+1)+(\alpha -1) \nu  n_2^2-\alpha  n_2 (2 \alpha +\nu ))+2 (k_1 (2 \alpha ^2-3 \alpha \\+\alpha  n_2+1) (-2 \alpha ^2-2 a \alpha  b+a b+\alpha ^2 n_2 )+\alpha \nu (-a b-2 \alpha +n_2 (\alpha -(\alpha -1) n_2)))=0.
		\end{multline*}
		
	\end{itemize}
	
\end{proof}

\begin{cor}\label{corr2}
	Suppose that $G_1$, $G_2$ are two non-isomorphic $A_\alpha$-Isospectral $k$-regulars graphs. If $H=K_{a,b}$ then $G_1\lozenge H$ and $G_2\lozenge H$ also are non isomorphic $A_\alpha$-Isospectral graphs.
\end{cor}
\section{$\alpha$-eigenvalue of $R$-vertex coranae}

Let $V(G_1)=\{u_1,...,u_{n_1}\}$, $I(G)=\{w_1,...,w_{m_1}\}$ and $V(G_2)=\{v_1,...,v_{n_1}\}$. For $i=1,...,n_1$, let $V^i(G_2)=\{v^i_1,...,v^i_{n_2}\}$ be the vertex set of the $i^{th}$ copy of $G_2$. Then $V(G_1\odot G_2)$ has a partition
\begin{equation}\label{partition}
V(G_1)\cup I(G_1)\cup \biggl(\bigcup_{i=1}^{n_1}V^i(G_2)\biggr).
\end{equation}
So the degrees of the vertices of $G_1\odot G_2$ are:
\begin{align*}
d_{G_1\odot G_2}(u_i) &=2d_{G_1}(v_i)+n_2 for i=1,...,n_1,\\
d_{G_1\odot G_2}(w_i) &=2 for  i=1,...,m_1,\\
d_{G_1\odot G_2}(u^i_j) &=d_{G_2}(v_j)+1 for i=1,...,n_1 and for j=1,...,n_2.
\end{align*}
Considering the $\alpha$-eigenvalue of this graph. The equation(\ref{mainequation}), of $G=G_1\odot G_2$ with respect to the partition (\ref{partition})  can be written as fallows,
$$A_\alpha(G)=\begin{bmatrix}
A_\alpha(G_1) & \beta B_1 & \beta I_{n_1}\otimes \mathbf{j^T_{n_2}}  \\
\beta B_1^T& 2\alpha I_{m_1} &\mathbf{O_{m_1\times n_1n_2}} \\
\beta I_{n_1}\otimes \mathbf{j_{n_2}}& \mathbf{O_{m_1\times n_1n_2}}  &I_{n_1}\otimes A_\alpha(G_2)
\end{bmatrix},$$
where $\beta=1-\alpha$
\begin{theorem}\label{vertexcorna}
	Let $G_1$ be an $k_1$-regular graph with $n_1$ vertices, and $G_2$ be an arbitrary graph with $n_2$ vertices. Then the characteristics polynomial of $G=G_1\odot G_2$ is
	\begin{multline*}
	$$f_{A_\alpha(G)}(\nu)=(\nu-2\alpha)^{m_1-n_1} \cdot \prod_{i=1}^{n_2}(\nu-\alpha-\nu_i(G_2))^{n_1} \cdot\\ \prod_{i=1}^{n_1}\Biggl(\nu^2-(k' +\beta^2\Gamma_{A_\alpha(G_2)}(\nu-k') +2\alpha+\nu_i(G_1))\nu\\+2\alpha k'+2\alpha\beta^2\Gamma_{A_\alpha(G_2)}(\nu-k')+k'\beta(2\alpha-1)-\beta \nu_i(G_1)\Biggr),$$
	\end{multline*}
	where $k'=\alpha(k_1+n_2)$.	
\end{theorem}
\begin{proof}
	Considering the partition  (\ref{partition}), we have
	\begin{align*}
	A_\alpha(G) &=\begin{bmatrix}
	A_\alpha(G_1)+\alpha(k_1+n_2)I_{n_1} & \beta B_1 & \beta I_{n_1}\otimes \mathbf{j^T_{n_2}}  \\
	\beta B_1^T& 2\alpha I_{m_1} &\mathbf{O_{m_1\times n_1n_2}} \\
	\beta I_{n_1}\otimes \mathbf{j_{n_2}}& \mathbf{O_{m_1\times n_1n_2}}  &I_{n_1}\otimes (A_\alpha(G_2)+\alpha I_{n_2})
	\end{bmatrix}.\\
	f_{A_\alpha(G)}(\nu) &=det\begin{bmatrix}
	(\nu-\alpha(k_1+n_2))I_{n_1}-A_\alpha(G_1) & -\beta B_1 & -\beta I_{n_1}\otimes \mathbf{j^T_{n_2}}  \\
	-\beta B_1^T& (\nu-2\alpha) I_{m_1} &\mathbf{O_{m_1\times n_1n_2}} \\
	-\beta I_{n_1}\otimes \mathbf{j_{n_2}}& \mathbf{O_{m_1\times n_1n_2}}  &I_{n_1}\otimes ((\nu-\alpha I_{n_2})-A_\alpha(G_2))
	\end{bmatrix},\\
	&=det(I_{n_1}\otimes ((\nu-\alpha I_{n_2})-A_\alpha(G_2)) )\cdot det S\\
	&=\prod_{i=1}^{n_2}(\nu-\alpha-\nu_i(G_2))^{n_1} \cdot det S.
	\end{align*}
	Where \\
	\begin{multline*}
	S=\begin{bmatrix}
	(\nu-\alpha(k_1+n_2))I_{n_1}-A_\alpha(G_1) & -\beta B_1 &\\
	-\beta B_1^T& (\nu-2\alpha) I_{m_1}
	\end{bmatrix}\\-\begin{bmatrix}
	\beta I_{m_1}\otimes \mathbf{j^T_{n_2}}\\
	\boldsymbol{O_{m_1\times n_1n_2}}
	\end{bmatrix}(I_{n_1}\otimes((x-\alpha)I_{n_2}-A_\alpha(G_2)))^{-1}\begin{bmatrix}
	\beta I_{m_1}\otimes \mathbf{j_{n_2}} & \boldsymbol{O_{m_1\times n_1n_2}}
	\end{bmatrix},
	\end{multline*}
	Assume $\alpha(k_1+n_{2})=k'$
	\begin{align*}
	S=\begin{bmatrix}
	(\nu-k')I_{n_1}-A_\alpha(G_1) & -\beta B_1 &\\
	-\beta B_1^T& (\nu-2\alpha) I_{m_1}
	\end{bmatrix}-\begin{bmatrix}
	\beta^2I_{n_1}\otimes\mathbf{j^T}(\nu-k'I_{n_2}-A_\alpha(G_2))^{-1}\mathbf{j} & \mathbf{O}\\
	\mathbf{O}&\mathbf{O}
	\end{bmatrix},
	\end{align*}
	Using Remark(\ref{remark4}) and equation (\ref{coronal}) we have,
	\begin{align*}
	S &=\begin{bmatrix}
	(\nu-k'-\beta^2\Gamma_{A_\alpha(G_2)}(\nu-k'))I_{n_1}-A_\alpha(G_1) & -\beta B_1\\
	-\beta B_1^T & (\nu-2\alpha)I_{m_1}
	\end{bmatrix},\\
	det S &= (x-2\alpha)^{m_1}det((\nu-k'-\beta^2\Gamma_{A_{\alpha}(G_2)}(\nu-k'))I_{n_1}-A_\alpha(G_2)-\frac{\beta^2 B_1B_1^T}{\nu-2\alpha}),\\
	&=(x-2\alpha)^{m_1}det((\nu-k'-\beta^2\Gamma_{A_{\alpha}(G_2)}(\nu-k')+\frac{k_1\beta(2\alpha-1)}{\nu-2\alpha})I_{n_1}-(1+\frac{\beta}{\nu-2\alpha})A_\alpha(G_1)),
	\end{align*}
	\begin{multline*}
	=(x-2\alpha)^{m_1-n_1}\cdot \prod_{i=1}^{n_1}(\nu^2-(k' +\beta^2\Gamma_{A_\alpha(G_2)}(\nu-k')\\+2\alpha+\nu_i(G_1))\nu+2\alpha k'+2\alpha\beta^2\Gamma_{A_\alpha(G_2)}(\nu-k')+k'\beta(2\alpha-1)-\beta \nu_i(G_1)).
	\end{multline*}
	Now the result fallows easily.	
\end{proof}	
\begin{cor}
	Let $G_1$ be an $k_1$-regular graph with $n_1$ vertices,and $G_2$ be an $k_2$-regular graph with $n_2$ vertices. Then
	\begin{multline*}
	f_{A_\alpha(G)}(\nu)=(\nu-2\alpha)^{m_1-n_1} \cdot \prod_{i=1}^{n_2}((\nu-\alpha-\nu_i(G_2))^{n_1}) \cdot\\ \prod_{i=1}^{n_1}\Biggl(\nu^3-(2k' +k_2+2\alpha+\nu_i(G_1))\nu^2+((k'+k_2)(k'+2\alpha+\nu_i(G_1))-n_2\beta^2+2\alpha k'\\+\beta(k_1(2\alpha-1)-\nu_i(G_2)))\nu-(k'+k_2)(2\alpha k'+k_1\beta+\beta\nu_i(G_2))+2\alpha\beta^2k'\Biggr),
	\end{multline*}
	where $k'=\alpha(k_1+n_2)$.
	
\end{cor}
\begin{proof}
	Using Theorem(\ref{vertexcorna}) and equation(\ref{coronal2}) it is obvious.
\end{proof}
\begin{cor}\label{corr3}
	\begin{itemize}
		\item[i.] If $G_1$ and $G_2$ are $\alpha$-Isospectral regular graphs, and $H$ is an arbitrary	graph, then $G_1 \odot H$ and $G_2\odot H$ are $\alpha$-Isospectral.
		\item[ii.] If $G$ is a regular graph, and $H_1$ and $H_2$ are $\alpha$-Isospectral graphs with $\Gamma_{A_{\alpha}(H_1)}(\nu)=\Gamma_{A_{\alpha}(H_2)}(\nu)$, then $G_1 \odot H$ and $G_2\odot H$ are $\alpha$-Isospectral.
		\item[iii.] If $G_1$ and $G_2$ are $\alpha$-Isospectral regular graphs, and $H_1$ and $H_2$ are $\alpha$-Isospectral graphs with $\Gamma_{A_{\alpha}(H_1)}(\nu)=\Gamma_{A_{\alpha}(H_2)}(\nu)$, then $G_1 \odot H$ and $G_2\odot H$ are $\alpha$-Isospectral.
	\end{itemize}
\end{cor}
\section{$\alpha$-eigenvalue of $R$-edge coronae}

Let $V(G_1)=\{u_1,...,u_{n_1}\}$, $I(G)=\{w_1,...,w_{m_1}\}$ and $V(G_2)=\{v_1,...,v_{n_1}\}$. For $i=1,...,m_1$, let $V^i(G_2)=\{v^i_1,...,v^i_{n_2}\}$ be the vertex set of the $i^{th}$ copy of $G_2$. Then $V(G_1\odot G_2)$ has a partition
\begin{equation}\label{partition1}
V(G_1)\cup I(G_1)\cup \biggl(\bigcup_{i=1}^{m_1}V^i(G_2)\biggr).
\end{equation}
So the degrees of the vertices of $G_1\odot G_2$ are:
\begin{align*}
d_{G_1\circleddash G_2}(u_i) &=2d_{G_1}(v_i) for i=1,...,n_1,\\
d_{G_1\circleddash G_2}(w_i) &=2+n_2 for  i=1,...,m_1,\\
d_{G_1\circleddash G_2}(u^i_j) &=d_{G_2}(u_j)+1 for i=1,...,m_1 and for j=1,...,n_2.
\end{align*}
Considering the $\alpha$-eigenvalue of this graph. The equation(\ref{mainequation}), of $G=G_1\circleddash G_2$ with respect to the partition (\ref{partition1})  can be written as fallows,
$$A_\alpha(G)=\begin{bmatrix}
A_\alpha(G_1)+\alpha k_1I_{n_1} & \beta B_1 & \mathbf{O_{n_1\times m_1n_2}}  \\
\beta B_1^T& (n_2+2)\alpha I_{m_1} & \beta I_{m_1}\otimes \mathbf{j^T_{n_2}} \\
\mathbf{O_{m_1n_2\times n_1}}& \beta I_{m_1}\otimes \mathbf{j_{n_2}}  &I_{m_1}\otimes (A_\alpha(G_2)+\alpha I_{n_1})
\end{bmatrix},$$
where $\beta=1-\alpha$.
\begin{theorem}\label{Redgecorona}
	Let $G_1$ be an $k_1$-regular graph with $n_1$ vertices and $m_1$ edges, and $G_2$ be an arbitrary graph with $n_2$ vertices. Then
	\begin{multline*}
	f_{A_\alpha(G)}(\nu)=(\nu-r)^{m_1-n_1} \cdot \prod_{i=1}^{n_2}(\nu-\alpha-\nu_i(G_2))^{m_1}\cdot\\ \prod_{i=1}^{n_1}(\nu^2-(r+\alpha k_1+\nu_i(G_1))\nu+(\alpha r+\beta(2\alpha-1)) k_1+(r -\beta)\nu_i(G_1)),
	\end{multline*}
	where $(n_2+2)\alpha+\Gamma_{A_\alpha(G_2)}(\nu-\alpha)=r$.
\end{theorem}
\begin{proof}
	Considering the partition (\ref{partition1})
	\begin{align*}
	f_{A_\alpha(G)}(\nu)&=\begin{bmatrix}
	(\nu-\alpha k_1)I_{n_1} A_\alpha(G_1) & -\beta B_1 & \mathbf{O_{n_1\times m_1n_2}}  \\
	-\beta B_1^T& (\nu-(n_2+2)\alpha) I_{m_1} & -\beta I_{m_1}\otimes \mathbf{j^T_{n_2}} \\
	\mathbf{O_{m_1n_2\times n_1}}& -\beta I_{m_1}\otimes \mathbf{j_{n_2}}  &I_{m_1}\otimes ((\nu-\alpha) I_{n_1}-A_\alpha(G_2))
	\end{bmatrix}\\
	&=det(I_{m_1}\otimes ((\nu-\alpha) I_{n_1}-A_\alpha(G_2)))\cdot detS,\\
	&=\prod_{i=1}^{n_2}(\nu-\alpha-\nu_i(G_2))^{m_1}\cdot det S.
	\end{align*}
	Where\\
	\begin{multline*}
	S=\begin{bmatrix}
	(\nu-\alpha k_1)I_{n_1}-A_\alpha(G_1) &-\beta B_1 \\
	\beta B_1^T& (\nu-(n_2+2)\alpha)I_{m_1}
	\end{bmatrix}\\-\begin{bmatrix}
	\mathbf{O_{n_1\times m_1n_2}}\\ \beta I_{m_1}\otimes \mathbf{j^T_{n_2}}
	\end{bmatrix}(I_{m_1}\otimes ((\nu-\alpha)I-A_\alpha(G_2))^{-1}) \begin{bmatrix}
	\mathbf{O_{m_1n_2\times n_1}} & \beta I_{m_1}\otimes \mathbf{j_{n_2}}
	\end{bmatrix}
	\end{multline*}
	\begin{align*}
	&=\begin{bmatrix}
	(\nu-\alpha k_1)I_{n_1}-A_\alpha(G_1) &-\beta B_1 \\
	\beta B_1^T& (\nu-(n_2+2)\alpha)I_{m_1}
	\end{bmatrix}-\begin{bmatrix}
	\mathbf{O} & \mathbf{O}\\
	\mathbf{O} & \Gamma_{A_\alpha(G_2)}(\nu-\alpha)I_{m_1}
	\end{bmatrix}\\
	&=\begin{bmatrix}
	(\nu-\alpha k_1)I_{n_1}-A_\alpha(G_1) &-\beta B_1 \\
	\beta B_1^T& (\nu-(n_2+2)\alpha+\Gamma_{A_\alpha(G_2)}(\nu-\alpha))I_{m_1}
	\end{bmatrix}
	\end{align*}
	Consider $(n_2+2)\alpha+\Gamma_{A_\alpha(G_2)}(\nu-\alpha)=r$. Using Remark(\ref{remark4}) and equation (\ref{coronal}) we have,
	\begin{align*}
	detS&=det((\nu-r)I_{m_1})\cdot det ((\nu-\alpha k_1)I_{n_1}-A_\alpha(G_1)-\frac{\beta B_1B_1^T}{\nu-r}), \\
	&=(\nu-r)^{m_1} \cdot det((\nu-\alpha k_1)I_{n_1}-A_\alpha(G_1)-\frac{\beta (A_\alpha(G_1)-(2\alpha-1)k_1I_{n_1})}{(\nu-r}),\\
	&=(\nu-r)^{m_1}\cdot det((\nu-\alpha k_1+\frac{\beta (2\alpha-1)k_1}{(\nu-r)})I_{n_1}-(1+\frac{\beta}{\nu-r})A_\alpha(G_2)),\\
	\end{align*}
	$=(\nu-r)^{m_1-n_1}\prod_{i=1}^{n_1}(\nu^2-(r+\alpha k_1+\nu_i(G_1))\nu+(\alpha r+\beta(2\alpha-1)) k_1+(r -\beta)\nu_i(G_1))$\\
	
	from which the results fallows.
\end{proof}
\begin{cor}\label{corr4}
	\begin{itemize}
		\item[i.] If $G_1$ and $G_2$ are $\alpha$-Isospectral regular graphs, and $H$ is an arbitrary graph, then $G_1\circleddash H$ and $G_2 \circleddash H$ are $\alpha$-Isospectral.
		\item[ii.] If $G$ is a regular graph, and $H_1$ and $H_2$ are $\alpha$-Isospectral with $\Gamma_{A_{\alpha}(H_1)}(\nu)=\Gamma_{A_{\alpha}(H_2)}(\nu)$, then $G\circleddash H_1$ and $G \circleddash H_2$ are $\alpha$-Isospectral.
	\end{itemize}
\end{cor}
\section*{Discussion} In \cite{XLSL}, the $A_\alpha$-characteristics polynomial for all graphs on at most $10$ vertices have been enumerated. Also counted the  number of graphs for which there exist at least one pair of non-isomorphic $\alpha$-Isospectral graphs. They get the smallest pair of non-isomorphic $\alpha$-Isospectral graphs with $9$ vertices, as an example see Figure(\ref{figure3}).
\begin{figure}[h!]
	\includegraphics[width=70mm,scale=0.7]{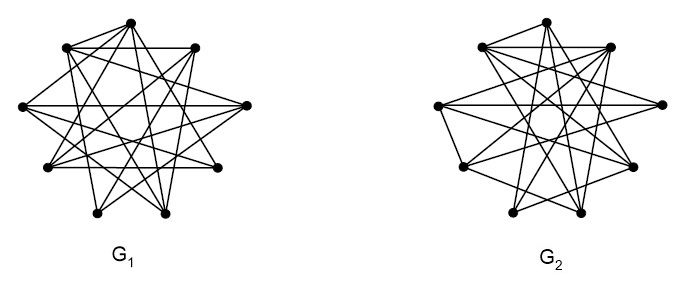}
	\centering
	\caption{Smallest pair of non-isomorphic $\alpha$-Isospectral graphs. }
	\label{figure3}
\end{figure}\\
Using the above mentioned graphs, in Corollary(\ref{corr1}),(\ref{corr2}),(\ref{corr3}) and Corollary(\ref{corr4})  we get  infinitely many non-isomorphic $\alpha$-Isospectral families of graphs.

{99}	
\end{document}